\documentclass[a4paper,11pt]{amsart}

\usepackage{epsfig}
\usepackage{amsthm,amsfonts}
\usepackage{amssymb,graphicx,color}
\usepackage{float}
\graphicspath{ {Images/} }
\usepackage[labelfont={bf},textfont={it}]{caption}
\usepackage[labelfont={bf, it}]{subcaption}
\usepackage[all]{xy}
\usepackage{verbatim}
\usepackage{hyperref}
\usepackage{enumitem}
\usepackage{tikz}
\usepackage{mathrsfs}

\newtheorem{theorem}{Theorem}[section]
\newtheorem*{theorem*}{Theorem}
\newtheorem{lemma}[theorem]{Lemma}
\newtheorem{corollary}[theorem]{Corollary}

\newtheorem{definition}[theorem]{Definition}
\newtheorem{conjecture}{Conjecture}
\newtheorem{claim}{Claim}[theorem]
\newtheorem{example}[theorem]{Example}
\newtheorem{remark}[theorem]{Remark}

\newtheorem{innercustomthm}{Theorem}
\newenvironment{customthm}[1]
  {\renewcommand\theinnercustomthm{#1}\innercustomthm}
  {\endinnercustomthm}

\newcommand{\R}{\mathbb{R}}
\newcommand{\C}{\mathbb{C}}

\renewcommand{\mod}{{\rm mod\,}}

\begin{document}

\title[Real and bi-Lipschitz versions of the Theorem of Nobile]
{Real and bi-Lipschitz versions of the Theorem of Nobile}

\author[J. E. Sampaio]{Jos\'e Edson Sampaio}

\address{Jos\'e Edson Sampaio: Departamento de Matem\'atica, Universidade Federal do Cear\'a,
	      Rua Campus do Pici, s/n, Bloco 914, Pici, 60440-900, 
	      Fortaleza-CE, Brazil. E-mail: {\tt edsonsampaio@mat.ufc.br} 	   
}

\thanks{The author was partially supported by CNPq-Brazil grant 303375/2025-6. This work was supported by the Serrapilheira Institute (grant number Serra -- R-2110-39576).
}


\begin{abstract}
The renowned Theorem of Nobile, proved by Nobile in 1975, states that a pure dimensional complex analytic set $X$ is analytically smooth if and only if its Nash transformation $\eta: \mathcal{N}(X) \to X$ is an analytic isomorphism. While the Theorem of Nobile was fundamental in complex geometry, it remained an open question for 50 years whether the theorem held for real analytic sets, even more so for cases that demand only $C^{k}$ smoothness. This paper presents a proof for the real version of the Theorem of Nobile, even under $C^{k}$ smoothness conditions. Specifically, we prove that for a pure dimensional real analytic set $X$ the following statements are equivalent:
\begin{enumerate}
 \item [(1)] $X$ is a real analytic (resp. $C^{k+1,1}$) submanifold;
 \item [(2)] the mapping $\eta\colon\mathcal{N}(X)\to X$ is a real analytic (resp. $C^{k,1}$) diffeomorphism;
 \item [(3)] the mapping $\eta\colon\mathcal{N}(X)\to X$ is a $C^{\infty}$ (resp. $C^{k,1}$) diffeomorphism;
 \item [(4)] $X$ is a $C^{\infty}$ (resp. $C^{k+1,1}$) submanifold.
\end{enumerate}
Consequently, we prove the bi-Lipschitz version of the Theorem of Nobile, demonstrating that a complex analytic set $X$ is analytically smooth if and only if its Nash transformation $\eta\colon \mathcal{N}(X)\to X$ is a homeomorphism that is locally bi-Lipschitz. A sharp version of this theorem, which holds in the much more general setting of locally definable sets in an o-minimal structure, is also presented here.
\end{abstract}

\maketitle

\tableofcontents

\section{Introduction}

Given a pure $d$-dimensional $\mathbb{K}$-analytic set $X$ in $\mathbb{K}^n$, $\mathcal{N}(X)$ is the closure in $X\times Gr_{\mathbb{K}}(d,n)$ of the graph of the {\bf Gauss mapping} $\nu\colon X\setminus {\rm Sing}(X)\to Gr_{\mathbb{K}}(d,n)$ given by $\nu(x)=T_xX$, where $\mathbb{K}$ is $\C$ or $\R$ and $Gr_{\mathbb{K}}(d,n)$ is the Grassmannian of $d$-dimensional $\mathbb{K}$-linear subspaces in $\mathbb{K}^n$. The set $\mathcal{N}(X)$ is called the {\bf Nash transformation of $X$} and the projection $\eta\colon \mathcal{N}(X)\to X$ is called the {\bf Nash mapping of $X$}. Sometimes we also call $\eta\colon \mathcal{N}(X)\to X$ the Nash transformation of $X$. The notion of Nash transformation trivially extends to algebraic varieties over a field of arbitrary characteristic.

An important problem in the resolution of singularities related to Nash transformations is the following conjecture:

\begin{conjecture}
A finite succession of Nash transformations resolves the singularities of $X$.
\end{conjecture}
This problem was proposed by Nash in a private communication to Hironaka in the early sixties (cf. \cite[p. 412]{Spivakovsky:1990}) and it was also posed by Semple in \cite{Semple:1954}. This is why we call it Nash-Semple's conjecture.

It is not the purpose of this article to discuss this conjecture, but it is important to say that a relevant partial answer to the above conjecture was given by Spivakovsky in \cite{Spivakovsky:1990}. He showed that in the case of surfaces, the normalised Nash transformations (i.e. Nash transformations followed by normalizations) resolve the singularities.
Recently, in the preprint \cite{CastilloDLL:2024}, the authors gave a very important partial answer to Nash-Semple's conjecture. They presented counterexamples for this conjecture for any dimension $d\geq 4$. Thus, Nash-Semple's conjecture is only an open problem in dimensions 2 and 3. However, there are still several important families of analytic sets for which this conjecture is an open problem, for example, hypersurfaces in all dimensions $d\geq 2$.

Probably the most basic and fundamental result related to the Nash transformation and Nash-Semple's conjecture is the following result proved by Nobile \cite{Nobile:1975}:
\begin{theorem}[Theorem of Nobile]\label{thm:Nobile}
Let $X$ be a pure $d$-dimensional analytic subset in $\C^n$. Then, $X$ is analytically smooth if and only if $\eta\colon \mathcal{N}(X)\to X$ is an analytic isomorphism.
\end{theorem}

In fact, the precise statement of the Theorem of Nobile is the following: {\it Let $\bar{k}$ be an algebraically closed field of characteristic zero (resp. $\bar{k} = \C$), $X$ be a pure $d$-dimensional algebraic (resp. analytic) variety over $\bar{k}$. Then, $X$ is smooth if and only if $\eta\colon \mathcal{N}(X)\to X$ is an isomorphism.}

The proof in \cite{Nobile:1975} shows that it suffices to prove the theorem in the analytic case with $\bar{k} = \C$, as stated in Theorem \ref{thm:Nobile}. The proof presented in \cite{Nobile:1975} is an analytic proof and systematically uses the complex structure. A more algebraic proof of the Theorem of Nobile is due to Teissier \cite{Teissier:1977} (cf. \cite{Nobile:2024}). 

The version of the Theorem of Nobile in positive characteristic does not work. Indeed, it is shown in \cite{Nobile:1975} that for the cusp $X:\, x^2=y^3 $, working in characteristic two, the Nash mapping $\eta\colon \mathcal{N}(X)\to X$ is an isomorphism, although, of course, $X$ is singular.
Recently, in the article \cite{DuarteN-B:2022}, the authors showed that if $X$ is a normal variety of any dimension, $\eta\colon \mathcal{N}(X)\to X$ is an isomorphism if and only if $X$ is smooth, even if the base field has positive characteristic. More about the Theorem of Nobile can be found in \cite{Nobile:2024}.

However, as far as the author knows, there is no proof for the real analytic version of this theorem, even more so when one only asks for $C^{k}$ smoothness. So, it is natural to ask: {\it Does the real version of the Theorem of Nobile hold true?}

While the Theorem of Nobile was fundamental in complex geometry, the real case remained an open question for 50 years.

Another natural problem is whether the analytic diffeomorphism condition of the Nash mapping can be weakened.

In this article, we completely solve these two questions. Indeed, we present a novel approach to the desingularisation of analytic sets through Nash transformations. With this approach, we prove the real version of the Theorem of Nobile. Notably, our method also demonstrates that, in the complex case, it suffices to simply require that the Nash mapping be locally bi-Lipschitz.

Recall that we have a distance defined in $Gr_{\R}(d,n)$ as follows: Given two linear subspaces $L_1$ and $L_2$ in $Gr_{\R}(d,n)$, 
\begin{eqnarray*}
\delta(L_1,L_2)&=& d_H(S_1,S_2)\\
            &=&\max\{\sup \limits_{a\in S_1}d(a,S_2),\sup \limits_{b\in S_2}d(b,S_1)\},
\end{eqnarray*}
where $S_i=L_i\cap \mathbb{S}^{n-1}$, $i=1,2$, and $d(a,S)=\inf\{\|a-x\|;x\in S\}$. For any set $X\subset \R^n$ that is pure $d$-dimensional and locally definable in an o-minimal structure on $\R$, we denote by $\delta_X$ the distance on $X\times Gr_{\R}(d,n)$ given by $\delta_X((x_1,L_1),(x_2,L_2))=\max \{\|x_1-x_2\|,\delta(L_1,L_2)\}$. So, $\mathcal{N}(X)$ is endowed with the distance induced by $\delta_X$.

So, we obtain the following version of the Theorem of Nobile.

\begin{customthm}{\ref{thm:metric_nobile}}
Let $X\subset \mathbb{C}^n$ be a pure dimensional complex analytic set. Then, $X$ is analytically smooth if and only if $\eta\colon \mathcal{N}(X)\to X$ is a homeomorphism that is locally bi-Lipschitz.
\end{customthm}

We also prove that $X$ is analytically smooth if and only if $\eta\colon \mathcal{N}(X)\to X$ is a homeomorphism such that $\eta^{-1}$ is (locally) $\alpha$-H\"older for all $\alpha\in (0,1)$ (see Corollary \ref{cor:c11_smooth_holder}).

We prove a real version of the Theorem of Nobile, when one only asks for $C^{k,1}$ smoothness (see Definitions \ref{defi:regulaties_one}, \ref{defi:regulaties_two}, and \ref{defi:regulaties_three}). Specifically, we prove the following result:

\begin{customthm}{\ref{thm:real_analytic_ck_smooth}}
Let $X\subset \R^n$ be a pure dimensional real analytic set, and let $k$ be a nonnegative integer number. Then, $X$ is a $C^{k+1,1}$ submanifold if and only if the mapping $\eta\colon\mathcal{N}(X)\to X$ is a homeomorphism such that $\eta^{-1}$ is $C^{k,1}$ smooth.
\end{customthm}

Consequently, we also obtain the following real version of the Theorem of Nobile.

\begin{customthm}{\ref{thm:real_analytic_smooth}}
Let $X\subset \R^n$ be a pure dimensional real analytic set. Then the following statements are equivalent:
\begin{enumerate}
 \item [(1)] $X$ is a real analytic submanifold;
 \item [(2)] the mapping $\eta\colon\mathcal{N}(X)\to X$ is a real analytic diffeomorphism;
 \item [(3)] the mapping $\eta\colon\mathcal{N}(X)\to X$ is a $C^{\infty}$ diffeomorphism;
 \item [(4)] $X$ is a $C^{\infty}$ submanifold.
\end{enumerate}
\end{customthm}

In fact, we obtain a stronger result.
Here, we present a proof that works even in the case of sets that are locally definable in an o-minimal structure under small assumptions (to know more about o-minimal geometry, see, for instance, \cite{Coste:1999} and \cite{Dries:1998}). 
For a subset $X\subset \R^n$, which is a $d$-dimensional set definable in an o-minimal structure, we say that $C_3(X)$ and $C_4(X)$ coincide linearly if for any $p\in X$ such that $C_4(X,p)$ is a $d$-dimensional linear subspace, then $C_3(X,p)=C_4(X,p)$.

We prove the following more general result:
\begin{customthm}{\ref{c11_smooth}}
Let $X\subset \R^n$ be a locally closed set that is pure $d$-dimensional and locally definable in an o-minimal structure on $\R$. Then, for a fixed nonnegative integer number $k$, $X$ is a $C^{k+1,1}$ smooth submanifold if and only if $C_3(X)$ and $C_4(X)$ coincide linearly and the mapping $\eta\colon\mathcal{N}(X)\to X$ is a homeomorphism such that $\eta^{-1}$ is $C^{k,1}$ smooth.  Moreover, $X$ is a $C^{\infty}$ submanifold if and only if the mapping $\eta\colon\mathcal{N}(X)\to X$ is a $C^{\infty}$ diffeomorphism. 
\end{customthm}

In fact, Theorem \ref{c11_smooth} holds true even for log-Lipschitz regularity (see Theorem \ref{c11_smooth_log-Lip}); however, surprisingly (and in contrast to Corollary \ref{cor:c11_smooth_holder} and the paper \cite{Sampaio:2025}), it does not hold true when we only require that $\eta$ is a homeomorphism such that $\eta^{-1}$ is $C^{0,\alpha}$ smooth for all $\alpha\in (0,1)$ (see Example \ref{exam:holder_fails}). Many other examples are also presented in order to show the sharpness of this result.

We prove Theorem \ref{c11_smooth} in Subsection \ref{subsec:main_results}. We prove this theorem when $k=0$ using the tools of o-minimal geometry and Lipschitz geometry in an involved way. For the case $k>0$, we use a bootstrap argument to show that if $X$ is a $C^{s,1}$ smooth submanifold for some $1\leq s\leq k$ and $\eta^{-1}$ is $C^{k,1}$ smooth, then $X$ is a $C^{s+1,1}$ smooth submanifold (see Claim \ref{claim:regularity}).

We finish this Introduction by noting that an important step in Nobile's proof in \cite{Nobile:1975} is the fact that two isomorphic complex analytic sets have isomorphic Nash transformations. However, this cannot be applied to real analytic sets that are bi-Lipschitz homeomorphic. Indeed, $X=\{(x,y,z)\in \R^3; x^{2025}+y^{2025}+z^{2025}=0\}$ and $\R^2$ are bi-Lipschitz homeomorphic (see \cite[Proposition 1.9]{Sampaio:2020}), but $\mathcal{N}(X)$ and $\mathcal{N}(\R^2)$ are not bi-Lipschitz homeomorphic.

\bigskip

 \noindent{\bf Acknowledgements}. The author thanks Alexandre Fernandes for his interest in this research and for the several interesting conversations that helped the author to improve the results of the paper. The author also thanks Bernard Teissier, Daniel Duarte, Euripedes da Silva, and Lev Birbrair for their interest in this research.

 \noindent{\bf Added in proof}. After this article was finished, Bernard Teissier informed the author of the interesting work by Oneto and Zatini \cite{OnetoZ:1991}, in which they proved the real algebraic version of the Theorem of Nobile. Specifically, in the case where $k$ is a characteristic zero field and $X$ is a pure dimensional algebraic variety over $k$, they showed that $X$ is smooth if and only if $\eta\colon \mathcal{N}(X)\to X$ is an (algebraic) isomorphism (see \cite[Corollary 4.2]{OnetoZ:1991}). Note that the real analytic case is not proved in \cite{OnetoZ:1991}.

\section{Preliminaries}\label{sec:preliminaries}

All the subsets of $\R^n$ or $\mathbb{C}^n$ considered in the paper are supposed to be equipped with the Euclidean distance. When we consider other distances, this is clearly emphasised.

\subsection{O-minimal structures}
\begin{definition}\label{definivel}
A structure on $\R$ is a collection $\mathcal{S}=\{\mathcal{S}_n\}_{n\in \mathbb{Z}_{>0}}$ where each $\mathcal{S}_n$ is a set of subsets of $\R^n$, satisfying the following axioms:
\begin{itemize}
\item [1)] All algebraic subsets of $\R^n$ are in $\mathcal{S}_n$;
\item [2)] For every $n$, $\mathcal{S}_n$ is a Boolean subalgebra of the powerset of $\R^n$;
\item [3)] If $A\in \mathcal{S}_m$ and $B \in S_n$, then $A \times B \in \mathcal{S}_{m+n}$.
\item [4)] If $\pi \colon \R^{n+1} \to \R^n$ is the projection on the first $n$ coordinates and $A\in \mathcal{S}_{n+1}$, then $\pi(A)\in \mathcal{S}_n$.
\end{itemize}
An element of $\mathcal{S}_n$ is called {\bf definable in $\mathcal{S}$}.
The structure $\mathcal{S}$ is said {\bf o-minimal} if it satisfies the following condition:
\begin{itemize}
\item [5)] The elements of $\mathcal{S}_1$ are precisely finite unions of points and intervals.
\end{itemize}
\end{definition}

\begin{definition}
A mapping $f\colon A\subset \R^n \to \R^m$ is called {\bf definable in $\mathcal{S}$} if its
graph is an element of $\mathcal{S}_{n+m}$. 
\end{definition}

A structure $\mathcal{S}$ is said {\bf polynomially bounded} if it satisfies the following condition:
\begin{itemize}
\item [6)]  For every function $f\colon \R \to  \R$ that is definable in $\mathcal{S}$, there exists a positive integer number $N$ such that $|f(t)|\leq t^N$ for all sufficiently large positive $t$.
\end{itemize}

We say that a set $X\subset \R^n$ is {\bf locally definable in $\mathcal{S}$} if for any $x\in \overline{X}$ there is a neighbourhood $U\subset \R^n$ of $x$ such that $X\cap U$ is definable in $\mathcal{S}$.

Throughout this paper, we fix an o-minimal structure $\mathcal{S}$ on $\R$.

In the sequel, the adjective {\bf definable} denotes definable in $\mathcal{S}$.

We say that a set $X\subset \R^n$ is {\bf locally definable} if for any $x\in \overline{X}$ there is a neighbourhood $U\subset \R^n$ of $x$ such that $X\cap U$ is definable.

Along this text, we will need the two following lemmas, which are inspired by the results of the paper \cite{Birbrair:2008}.
\begin{lemma}\label{lem:estimate_derivative}
Let $\gamma \colon [0,\varepsilon) \to \R$ be a continuous definable function such that $g(0)=0$. Assume that the germ of $g$ at $0$ is not the germ of the zero function. Then
$$
\lim\limits_{t\to 0^+}\frac{g(t)}{g'(t)}=0.
$$
\end{lemma}
\begin{proof}
Since $g$ is definable, the $\lim\limits_{t\to 0^+}\frac{g(t)}{g'(t)}$ exists, maybe its $+\infty$ or $-\infty$. By L'Hospital rule, we have
$$
0=\lim\limits_{t\to 0^+}\frac{tg(t)}{g(t)}=\lim\limits_{t\to 0^+}\frac{tg'(t)+g(t)}{g'(t)}=\lim\limits_{t\to 0^+}\frac{g(t)}{g'(t)}.
$$
\end{proof}

\begin{lemma}\label{lem:estimate_derivative_log}
Let $\gamma \colon [0,\varepsilon) \to [0,+\infty)$ be a continuous definable function such that $g(0)=0$. Assume that the germ of $g$ at $0$ is not the germ of the zero function. Then
$$
\lim\limits_{t\to 0^+}\frac{g'(t)}{g(t)\log g(t)}=-\infty.
$$
\end{lemma}
\begin{proof}

By shrinking $\varepsilon$, if necessary, we may assume that $1>g(t)>0$ for all $t\in (0,\varepsilon)$.
Define $h,f\colon (0,\varepsilon) \to \R$ by $h(t)=\log g(t)$ and $f(t)=-\frac{1}{h(t)}$.

We have that $f$ is definable, $f(t)>0$ for all $t\in (0,\varepsilon)$ and $\lim\limits_{t\to 0^+} f(t)=0$. Lemma \ref{lem:estimate_derivative},
$$
\lim\limits_{t\to 0^+}\frac{f'(t)}{f(t)}=+\infty.
$$
However,
$$
\frac{f'(t)}{f(t)}=-\frac{h'(t)}{h(t)}=-\frac{g'(t)}{g(t)\log g(t)}.
$$
Therefore,
$$
\lim\limits_{t\to 0^+}\frac{g'(t)}{g(t)\log g(t)}=-\lim\limits_{t\to 0^+}\frac{f'(t)}{f(t)}=-\infty.
$$
\end{proof}

\subsection{Tangent cones}
\begin{definition}\label{def:tg_cone}
Let $X\subset \R^{n}$ be a subset and let $p\in \overline{X}$.
We say that $v\in \R^{n}$ is {\bf a tangent vector of $X$ at $p$} if there are a sequence $\{x_j\}_{j\in \mathbb{N}}\subset X$ and a sequence of positive real numbers $\{t_j\}_{j\in \mathbb{N}}$ such that $\lim\limits_{j\to \infty} t_j=0$ and $\lim\limits_{j\to \infty} \frac{1}{t_j}(x_j-p)=v$. The set of all tangent vectors of $X$ at $p$ is denoted by $C_3(X,p)$ and is called the {\bf tangent cone of $X$ at $p$}.
\end{definition}
\begin{remark}
Let $X\subset \mathbb{C}^n$ be a complex analytic (resp. algebraic) set and $x_0\in X$. In this case, $C(X,x_0)$ (resp. $C(X, \infty)$) is the zero set of a set of complex homogeneous polynomials, as we can see in \cite[Theorem 4D]{Whitney:1972}. In particular, $C(X,x_0)$ (resp. $C(X, \infty)$) is the union of complex lines passing through the origin $0\in\mathbb{C}^n$.
\end{remark}

\begin{definition}
Given a pure $d$-dimensional definable set $X\subset \R^n$ and $p\in X$, $C_4(X,p)$ is the union of the $d$-dimensional linear subspaces $T\subset \R^n$ such that there is a sequence $\{x_j\}_{j\in \mathbb{N}}\subset X\setminus Sing_1(X)$ satisfying $\lim\limits_{j\to +\infty}T_{x_j}X=T$, where $Sing_1(X)$ denotes the set of points $x\in X$ such that $X$ is not a $C^1$ submanifold around $x$.
\end{definition}

Note that $C_3(X,p)\subset C_4(X,p)$ for all $p\in X$.

\subsection{Multiplicity and degree of real sets}
This subsection is closely related to the subsection with the same title in \cite{FernandesJS:2022}.

Let $X\subset \R^{n}$ be a $d$-dimensional real analytic set with $0\in X$ and 
$$
X_{\C}= V(\mathcal{I}_{\R}(X,0)),
$$
where $\mathcal{I}_{\R}(X,0)$ is the ideal in $\mathbb{C}\{z_1,\ldots,z_n\}$ generated by the complexifications of all germs of real analytic functions that vanish on the germ $(X,0)$. We know that $X_{\C}$ is a germ of a complex analytic set and $\dim_{\C}X_{\C}=\dim_{\R}X$. Then, for a linear projection $\pi:\C^{n}\to\C^d$ such that $\pi^{-1}(0)\cap C(X_{\C},0) =\{0\}$, there exists an open neighbourhood $U\subset \C^n$ of $0$ such that $\# (\pi^{-1}(x)\cap (X_{\C}\cap U))$ is constant for a generic point $x\in \pi(U)\subset\C^d$. This number is the multiplicity of $X_{\C}$ at the origin and is denoted by $m(X_{\C},0)$.

\begin{definition}
With the above notation, we define the multiplicity of $X$ at the origin by $m(X,0):=m(X_{\C},0)$.
\end{definition}

More about the multiplicity of real analytic sets can be learnt in \cite{Sampaio:2020b}.

\begin{remark}\label{transversal-cones}
It follows from the works in \cite{Sampaio:2020b} and \cite{FernandesJS:2022} that if $X\subset \R^n$ is a pure $d$-dimensional real analytic set and $\pi\colon \R^n\to \R^d$ is a projection such that $\pi^{-1}(0)\cap C(X,0)=\{0\}$, then there are an open neighbourhood $U\subset \R^n$ of $0$ and a subanalytic set $\sigma\subset \R^d$ such that $\dim_{\R}\sigma <d$ and $m(X,0)\equiv \# (\pi^{-1}(t)\cap U)(\mod \, 2)$ for all small enough $t\in \R^d\setminus \sigma$, where $\# (\pi^{-1}(t)\cap U)$ denotes the cardinality of $\pi^{-1}(t)\cap U$.
\end{remark}

\subsection{Smooth mappings on sets}

In this subsection, we recall the definitions of smoothness of mappings defined only in subsets of manifolds. 

\begin{definition}\label{defi:regulaties_one}
Let $(M_1,d_1)$ and $(M_2,d_2)$ be two metric spaces. We say that a map $f\colon M_1\to M_2$ is:
\begin{itemize}
    \item {\bf $\alpha$-H\"older} if there is a constant $C>0$ such that 
    $$
    d_2(f(x),f(y))\leq Cd_1(x,y)^{\alpha} 
    $$
    for all $x,y\in M_1$. In this case, we also say that $f$ is {\bf $C^{0,\alpha}$ smooth}. When $f$ is 1-H\"older, we also say that $f$ is {\bf Lipschitz};
    \item {\bf $\gamma$-log-Lipschitz} if there is a constant $C>0$ such that
    $$
    d_2(f(x),f(y))\leq Cd_1(x,y)|\log (d_1(x,y))|^{\gamma} 
    $$
    for all $x,y\in M_1$ such that $d_1(x,y)\leq 1/2$.
    In this case, we also say that $f$ is {\bf $C^{0,\gamma-log-Lip}$ smooth}. When $f$ is $1$-log-Lipschitz, we simply say that $f$ is {\bf log-Lipschitz} or $f$ is {\bf $C^{0,log-Lip}$ smooth}.
\end{itemize}
When $M_1$ and $M_2$ are two $C^{\infty}$ smooth manifolds, $k\geq 0$ and $\beta\in [0,1]\cup \{\gamma-log-Lip;$ $\gamma\geq 0\}$, we say that a map $f\colon M_1\to M_2$ is $C^{k+1,\beta}$ smooth if its derivative $Df\colon TM_1\to TM_2$ is $C^{k,\beta}$ smooth.
\end{definition}

\begin{definition}\label{defi:regulaties_two}
Let $M_1$ and $M_2$ be two $C^{\infty}$ smooth manifolds. Let $f\colon A_1\to A_2$ be a mapping, where $A_1$ and $A_2$ are subsets of $M_1$ and $M_2$, respectively. We say that $f$ is $C^{k,\alpha}$ (resp. $C^{\infty}$) smooth if for each $p\in A_1$, there are an open neighbourhood $U\subset M_1$ of $p$ and a $C^{k,\alpha}$ (resp. $C^{\infty}$) smooth mapping $F\colon U\to M_2$ such that $F|_{A_1\cap U}=f|_{A_1\cap U}$. We say that $f$ is a $C^{k,\alpha}$ (resp. $C^{\infty}$) diffeomorphism if $f$ is a bijection and, moreover, $f$ and $f^{-1}$ are $C^{k,\alpha}$ (resp. $C^{\infty}$) smooth.
\end{definition}
\begin{definition}\label{defi:regulaties_three}
Let $M_1$ and $M_2$ be two real (resp. complex) analytic manifolds. Let $f\colon A_1\to A_2$ be a mapping, where $A_1$ and $A_2$ are subsets of $M_1$ and $M_2$, respectively. We say that $f$ is real (resp. complex) analytic smooth if for each $p\in A_1$, there are an open neighbourhood $U\subset M_1$ of $p$ and a real (resp. complex) analytic smooth mapping $F\colon U\to M_2$ such that $F|_{A_1\cap U}=f|_{A_1\cap U}$. We say that $f$ is a real (resp. complex) analytic diffeomorphism if $f$ is a bijection and, moreover, $f$ and $f^{-1}$ are real (resp. complex) analytic smooth.
\end{definition}

\section{Main results}\label{sec:main_results}

\subsection{Definable version of the Theorem of Nobile}\label{subsec:main_results}
\begin{theorem}\label{c11_smooth}
Let $X\subset \R^n$ be a locally closed set that is pure $d$-dimensional and locally definable in an o-minimal structure on $\R$. Then, for a fixed nonnegative integer number $k$, $X$ is a $C^{k+1,1}$ smooth submanifold if and only if $C_3(X)$ and $C_4(X)$ coincide linearly and the mapping $\eta\colon\mathcal{N}(X)\to X$ is a homeomorphism such that $\eta^{-1}$ is $C^{k,1}$ smooth.  Moreover, $X$ is a $C^{\infty}$ submanifold if and only if the mapping $\eta\colon\mathcal{N}(X)\to X$ is a $C^{\infty}$ diffeomorphism. 
\end{theorem}
\begin{proof}
It is clear that if $X$ is a $C^{k+1,1}$ smooth submanifold, then $C_3(X)$ and $C_4(X)$ coincide linearly and $\eta\colon\mathcal{N}(X)\to X$ is a homeomorphism such that $\eta^{-1}$ is $C^{k,1}$. Moreover, if $X$ is a $C^{\infty}$ smooth submanifold, then $\eta\colon\mathcal{N}(X)\to X$ is a $C^{\infty}$ smooth diffeomorphism.

Reciprocally, assume that the mapping $\eta\colon\mathcal{N}(X)\to X$ is a homeomorphism such that $\eta^{-1}$ is $C^{k,1}$ smooth and $C_3(X)$ and $C_4(X)$ coincide linearly.

Firstly, we assume $k=0$. In this case, $\eta\colon\mathcal{N}(X)\to X$ is a homeomorphism that is locally bi-Lipschitz.

Since $\eta\colon\mathcal{N}(X)\to X$ is in particular a bijection, we obtain that $C_4(X,p)$ is a $d$-dimensional linear subspace for all $p\in X$.
Since $C_3(X)$ and $C_4(X)$ coincide linearly, then $C_3(X,p)$ is a $d$-dimensional linear subspace for all $p\in X$ and continuously varies in $p$. For simplicity, in this case, for each $p\in X$, we denote $C_3(X,p)$ by $T_pX$.

\begin{claim}\label{claim:Lip_implies_c1}
$X$ is a $C^{1}$ submanifold.
\end{claim}
\begin{proof}[Proof of Claim \ref{claim:Lip_implies_c1}] 
Suppose on the contrary that $X$ is not a $C^1$ submanifold. By \cite[Theorem 4.4]{KurdykaL-GN:2018}, there exists $p \in X$ such that there is no neighbourhood of $p$ in $X$ to which the restriction to $X$ of the orthogonal projection $\pi_p\colon \R^n\to T_xX$ is injective. We may assume that $p$ coincides with the origin $0$ and we choose the linear coordinates $(x,y)$ of $\R^n$ such that $T_p X=\{(x,y);y=0\}= \mathbb{R}^d\times \{0\}\cong \mathbb{R}^d$. With this identification, $\pi_p$ now becomes the orthogonal projection to the first $k$ coordinates $\pi\colon  \mathbb{R}^n \to  \mathbb{R}^d$. 

By \cite[Lemma 4.1]{KurdykaL-GN:2018}, there is an open neighbourhood $U$ of $0$ in $X$ such that $\pi|_ U$ is an open map. By shrinking $U$, if necessary, we assume that $\pi(U) = \mathbf B^{d}(0,r)$ and for every $q \in U$, $\delta(T_q X, T_0 X)<\frac{1}{2}$ and, in particular, $T_q X$ is not orthogonal to $T_0 X$. Since $X$ is locally closed, we may assume that $r$ was taken small enough so that $X\cap \overline{\mathbf B^{n}(0,r)}\cap X$ is a compact set and $\pi|_U^{-1}(0)=\{0\}$. 

Let $N=\sup\limits_{x\in \mathbf B^{d}(0,r)} \#\pi|_{U} ^{-1}(x)$ and $S= \{ x\in \mathbf B^{d}(0,r): \#\pi|_{U} ^{-1}(x)= N\}$, where $\#\pi|_{U} ^{-1}(x)$ denotes the cardinality of $\pi|_{U} ^{-1}(x)$.  By shrinking $r$, if necessary, we assume that $0\in \overline{S}$. 
Given $x_0\in S\cap\mathbf B^{d}(0,r/2)$, let $t=\sup \{s\leq r/2; \#\pi|_{U} ^{-1}(x)=N$ for all $x\in \mathbf B^{d}(x_0,s)\}$. We have that $S$ is an open set, $\pi_U^{-1}(\mathbf B^{d}(x_0,t))$ has exactly $N$ connected components, say $X_1,...., X_N$, and each $X_i$ is the graph of a $C^1$ definable mapping $f_i\colon \mathbf B^{d}(x_0,t)\to \R^{n-d}$. By the assumptions on the tangent cones, we also have that each $f_i$ has bounded derivative, and thus it is a Lipschitz mapping and has a Lipschitz extension $\bar f_i\colon\overline{\mathbf B^{d}(x_0,t)}\to \R^{n-d}$. Note that each $\bar f_i$ is also a definable mapping. 
Thus, there is $x\in \overline {\mathbf B^{d}(x_0,t)}$ such that $\|x-x_0\|=t$ and $\#\pi|_{U} ^{-1}(x)<N$. This implies that there are $i$ and $j$ such that $\bar f_i(x)=\bar f_j(x)$. 

Let $\lambda\geq 1$ be a number such that $\|Df_i\|\leq \lambda$ and $\|Df_j\|\leq \lambda$ on $S$.

There is a constant $\tilde K$ such that $\delta(T_{(z,f_i(z))}X_i,T_{(z,f_j(z))}X_j)\geq \tilde K\|(Df_i)_z-(Df_j)_z\|$ for all $z\in S$, where $\|(Df_i)_z-(Df_j)_z\|:=\sup \{\|(Df_i)_z(u)-(Df_j)_z(u)\|;u\in \R^d,\,\|u\|\leq 1\}$.

In fact, for a fixed $z\in S$, let $S_i=T_{(z,f_i(z))}X\cap \mathbb{S}^{n-1}$ and $S_j=T_{(z,f_j(z))}X\cap \mathbb{S}^{n-1}$. Then, we have
\begin{eqnarray*}
\delta(T_{(z,f_i(z))}X,T_{(z,f_j(z))}X)&=& d_H(S_i,S_j)\\
            &=&\max\{\sup \limits_{a\in S_i}d(a,S_j),\sup \limits_{b\in S_j}d(b,S_i)\},
\end{eqnarray*}
and thus for any $u\in \R^d\setminus\{0\}$, there is $v\in \R^d\setminus\{0\}$ such that
\begin{eqnarray*}
\delta(T_{(z,f_i(z))}X,T_{(z,f_j(z))}X)&\geq &\left \|\frac{(u,(Df_i)_z(u))}{\|(u,(Df_i)_z(u))\|}-\frac{(v,(Df_j)_z(v))}{\|(v,(Df_j)_z(v))\|}\right\|.
\end{eqnarray*}
Denoting $\tilde u=\frac{u}{\|(u,(Df_i)_z(u))\|}$ and $\tilde v=\frac{v}{\|(v,(Df_j)_z(v))\|}$, we have
\begin{eqnarray*}
\delta(T_{(z,f_i(z))}X,T_{(z,f_j(z))}X)&\geq & \|(\tilde u,(Df_i)_z(\tilde u))-(\tilde v,(Df_j)_z(\tilde v))\|\\
                &\geq & \frac{\sqrt{2}}{2}(\|\tilde u-\tilde v\|+\|(Df_i)_z(\tilde u)-(Df_j)_z(\tilde v)\|)\\
                &\geq & \frac{\sqrt{2}}{2}(\|\tilde u-\tilde v\|+\frac{1}{\lambda}\|(Df_i)_z(\tilde u)-(Df_j)_z(\tilde v)\|)\\
                &\geq & \frac{\sqrt{2}}{2}\|\tilde u-\tilde v\|+\frac{\sqrt{2}}{2\lambda}\|(Df_i)_z(\tilde u)-(Df_j)_z(\tilde u)\|\\
                & & -\frac{\sqrt{2}}{2\lambda}\|(Df_j)_z(\tilde u-\tilde v)\|\\
                &\geq & \frac{\sqrt{2}}{2\lambda}\|(Df_i)_z(\tilde u)-(Df_j)_z(\tilde u)\|.                
\end{eqnarray*}
Note that $\|\tilde u\|\geq \frac{1}{\sqrt{1+\lambda^2}}$. Therefore,
\begin{eqnarray}\label{eq:tg_derivative}
\delta(T_{(z,f_i(z))}X,T_{(z,f_j(z))}X)&\geq & \tilde K\|(Df_i)_z-(Df_j)_z\|,               
\end{eqnarray}
where $\tilde K=\frac{\sqrt{2}}{2\lambda\sqrt{1+\lambda^2}}$.

Let $\gamma\colon [0,\epsilon)\to \overline {\mathbf B^{d}(x_0,t)}$ be the arc given by $\gamma(t)=x+tu$, where $u=\frac{x_0-x}{\|x_0-x\|}$. We have $\gamma((0,\epsilon))\subset \mathbf B^{d}(x_0,t)$ and $\|\gamma'(t)\|= \|u\|=1$ for all $t\in [0,\epsilon)$. Let $\gamma_i(t)=(\gamma(t),f_i(\gamma(t)))$ and $ \gamma_j(t)=(\gamma(t),f_j(\gamma(t)))$. For simplicity, denote $g(t)=f_i\circ\gamma(t)-f_j\circ\gamma(t)$. 

By the inequality \ref{eq:tg_derivative},
$$
\delta(T_{\gamma_i(t)}X,T_{\gamma_j(t)}X)\geq  \tilde K\|(Df_i)_{\gamma(t)}(\gamma'(t))-(Df_j)_{\gamma(t)}(\gamma'(t))\| 
$$
for all $t\in (0,\epsilon)$. Then
$$
\delta(T_{\gamma_i(t)}X,T_{\gamma_j(t)}X)\geq  \tilde K\|g'(t)\| 
$$
for all $t\in [0,\epsilon)$.

Thus, since $\eta$ is a local bi-Lipschitz homeomorphism, there is a constant $M>0$ such that
\begin{eqnarray*}
\|g(t)\|=\|\gamma_i(t)-\gamma_j(t)\|&=& \|f_i\circ\gamma(t)-f_i\circ\gamma(t)\|\\
              &\geq & \frac{1}{M}\delta_X(\eta^{-1}(\gamma_i(t)),\eta^{-1}(\gamma_j(t)))\\
              &=& \frac{1}{M}\max \{ \|\gamma_i(t)-\gamma_j(t)\|, \delta(T_{\gamma_i(t)}X_i,T_{\gamma_j(t)}X_j)\}\\
              &\geq & \frac{1}{M}\max \{\|\gamma_i(t)-\gamma_j(t)\|, \tilde K \|g'(t)\|\}\\
              &\geq & \frac{\tilde K}{M} \|g'(t)\|.
\end{eqnarray*}
There is $\ell$ such that $\|g(t)\|\leq \sqrt{n-d} |g_{\ell}(t)|$, and thus
$$
|g_{\ell}'(t)|\leq \sqrt{n-d}\frac{M}{\tilde K}  |g_{\ell}(t)|,
$$

Since $g_{\ell}$ is a definable function and $g_{\ell}(0)=0$, by exchanging the indices $i$ and $j$, if necessary, we may assume that $g_{\ell}(t)>0$ and $g'_{\ell}(t)>0$ for all small enough $t>0$.
By Lemma \ref{lem:estimate_derivative}, $\lim\limits_{t\to 0^+}\frac{g_{\ell}(t)}{g'_{\ell}(t)}=0$, which gives a contradiction.

Therefore, $X$ is a $C^{1}$ submanifold.
\end{proof}

\begin{claim}\label{claim:c1_implies_c11}
$X$ is a $C^{1,1}$ submanifold.
\end{claim}
\begin{proof}[Proof of Claim \ref{claim:c1_implies_c11}]
This follows using the Pl\"ucker embedding, but we present a direct proof here.

Since we are dealing with a local problem, we may assume that $X$ is the graph of a $C^1$ mapping $f\colon \mathbf B^{d}(0,r)\to \R^{n-d}$, for some $r>0$. By shrinking $r>0$, if necessary, we may assume that $f$ is $\lambda$-Lipschitz. We assume that $\lambda\geq 1$.
Then for $S_1=T_{(z,f(z))}X\cap \mathbb{S}^{n-1}$ and $S_2=T_{(w,f(w))}X\cap \mathbb{S}^{n-1}$, we have
\begin{eqnarray*}
\delta(T_{(z,f(z))}X,T_{(w,f(w))}X)&=& d_H(S_1,S_2)\\
            &=&\max\{\sup \limits_{a\in S_1}d(a,S_2),\sup \limits_{b\in S_2}d(b,S_1)\},
\end{eqnarray*}
and thus for any $u\in \R^d\setminus\{0\}$, there is $v\in \R^d\setminus\{0\}$ such that
\begin{eqnarray*}
\delta(T_{(z,f(z))}X,T_{(w,f(w))}X)&\geq &\left \|\frac{(u,Df_z(u))}{\|(u,Df_z(u))\|}-\frac{(v,Df_w(v))}{\|(v,Df_w(v))\|}\right\|.
\end{eqnarray*}
Denoting $\tilde u=\frac{u}{\|(u,Df_z(u))\|}$ and $\tilde v=\frac{v}{\|(v,Df_w(v))\|}$, we have
\begin{eqnarray*}
\delta(T_{(z,f(z))}X,T_{(w,f(w))}X)&\geq & \|(\tilde u,Df_z(\tilde u))-(\tilde v,Df_w(\tilde v))\|\\
                &\geq & \frac{\sqrt{2}}{2}(\|\tilde u-\tilde v\|+\|Df_z(\tilde u)-Df_w(\tilde v)\|)\\
                &\geq & \frac{\sqrt{2}}{2}(\|\tilde u-\tilde v\|+\frac{1}{\lambda}\|Df_z(\tilde u)-Df_w(\tilde v)\|)\\
                &\geq & \frac{\sqrt{2}}{2}(\|\tilde u-\tilde v\|+\frac{1}{\lambda}(\|Df_z(\tilde u)-Df_w(\tilde u)\|-\|Df_w(\tilde u-\tilde v)\|))\\
                &\geq & \frac{\sqrt{2}}{2\lambda}\|Df_z(\tilde u)-Df_w(\tilde u)\|.                
\end{eqnarray*}
Note that $\|\tilde u\|\geq \frac{1}{\sqrt{1+\lambda^2}}$. Therefore,
\begin{eqnarray*}
\delta(T_{(z,f(z))}X,T_{(w,f(w))}X)&\geq & \frac{\sqrt{2}}{2\lambda\sqrt{1+\lambda^2}}\|Df_z-Df_w\|,                
\end{eqnarray*}
where $\|Df_z-Df_w\|=\sup \{\|Df_z(u)-Df_w(u)\|;u\in \R^d,\,\|u\|\leq 1\}$.
By hypothesis, there is a constant $K>0$ such that $\delta(T_{(z,f(z))}X,T_{(w,f(w))}X)\leq K\|(z,f(z))-(w,f(w))\|$. Then,

\begin{eqnarray*}
\|Df_{z}-Df_{w}\|&\leq& K\lambda\sqrt{2+2\lambda^2}\|(z,f(z))-(w,f(w))\|\\
                 &\leq& K\lambda\sqrt{2+2\lambda^2}(\|z-w\|+\|f(z)-f(w)\|)\\
                 &\leq& K\lambda\sqrt{2+2\lambda^2}(1+\lambda)\|z-w\|.
\end{eqnarray*}
Therefore, $Df$ is Lipschitz, which shows that $X$ is a $C^{1,1}$ submanifold.
\end{proof}

Now, we assume that $k\geq 1$. By the first part of this proof, $X$ is a $C^{1,1}$ submanifold. 

Since our problem is a local problem, we may assume that $X$ is the graph of a $C^{1,1}$ smooth mapping $h\colon B\to \mathbb{R}^{n-d}$ for some open ball $B\subset \mathbb{R}^d$ and such that $h$ is a definable mapping. Let us write $h=(h_1,...,h_d)$.

We are going to show that $h$ is $C^{k+1,1}$. This follows from the following claim:
\begin{claim}\label{claim:regularity}
If $h$ is $C^{s,}$ smooth for some $1\leq s\leq k$, then $h$ is $C^{s+1,1}$ smooth.
\end{claim}
\begin{proof}[Proof of Claim \ref{claim:regularity}]
Assume that $h$ is $C^{s,1}$ smooth for some $1\leq s\leq k$. Then $X$ is a $C^{s,1}$ submanifold. Since $\eta\colon\mathcal{N}(X)\to X$ is a $C^{k,1}$ smooth diffeomorphism and $k\geq s$, then $\mathcal{N}(X)$ is a $C^{s,1}$ submanifold as well, and thus $\nu \colon X\to Gr_{\mathbb{R}}(d,n)$ is $C^{s,1}$ smooth. By using the Pl\"ucker embedding $p\colon Gr_{\mathbb{R}}(d,n)\to \mathbb{P}(\bigwedge^d\R^n)$, where $\bigwedge^d\R^n$ is the $d$-th exterior power of $\R^n$ and $\mathbb{P}(\bigwedge^d\R^n)$ is the projectivization of $\bigwedge^d\R^n$, we see that $1$ and each partial derivative $\frac{\partial h_i}{\partial x_j}$ are coordinates of $(id\times p)\circ \eta^{-1}(x,h(x))$, and thus $\frac{\partial h_i}{\partial x_j}$ is $C^{s,1}$ smooth. Therefore, $h$ is $C^{s+1,1}$ smooth.
\end{proof}
Since $h$ is $C^{1,1}$ smooth, by using recursively Claim \ref{claim:regularity}, $h$ is $C^{k+1,1}$ smooth.

Moreover, if $\eta\colon\mathcal{N}(X)\to X$ is a $C^{\infty}$ diffeomorphism, then by Claim \ref{claim:regularity}, $h$ is $C^{k,1}$ smooth for all positive integers $k$. Therefore, $h$ is $C^{\infty}$ smooth, and thus $X$ is a $C^{\infty}$ submanifold.
\end{proof}

Note that if $X\subset \R^n$ is a locally closed set that is pure $d$-dimensional and locally definable in a polynomially bounded o-minimal structure on $\R$, then $\eta^{-1}$ is locally definable in a polynomially bounded o-minimal structure on $\R$. Thus, if $\eta^{-1}$ is $C^{k,\alpha}$ smooth for all $\alpha\in (0,1)$, then $\eta^{-1}$ is $C^{k,1}$ smooth. Then, we have the following direct consequence.
\begin{corollary}\label{cor:c11_smooth_holder}
Let $X\subset \R^n$ be a locally closed set that is pure $d$-dimensional and locally definable in a polynomially bounded o-minimal structure on $\R$. Then, for a fixed nonnegative integer number $k$, $X$ is a $C^{k+1,1}$ smooth submanifold if and only if $C_3(X)$ and $C_4(X)$ coincide linearly and the mapping $\eta\colon\mathcal{N}(X)\to X$ is a homeomorphism such that $\eta^{-1}$ is $C^{k,\alpha}$ smooth for all $\alpha\in (0,1)$.
\end{corollary}

\begin{remark}
Let $X\subset \R^n$ be a locally closed set that is pure $d$-dimensional and locally definable in an o-minimal structure on $\R$. 
If $\eta\colon \mathcal{N}(X)\to X$ is a homeomorphism such that $\eta^{-1}$ is $C^{0,log-Lip}$ smooth, then $X$ is a $C^{1}$ smooth submanifold. Indeed, if $X$ is not $C^1$ smooth submanifold, let $f_i$, $f_j$, $\gamma_i$, $\gamma_j$ and $g$ as in the proof of Theorem \ref{c11_smooth}. Let us write $g(t)=(g_1(t),...,g_{n-d}(t))$. So, we can prove that there is a constant $C$ such that 
$$
|g_{\ell}'(t)|\leq \|g'(t)\|\leq -C\|g(t)\|\log(\|g(t)\|),
$$
for all $\ell \in \{1,...,n-d\}$. There is $\ell$ such that $\|g(t)\|\leq \sqrt{n-d} |g_{\ell}(t)|$, and thus
$$
|g_{\ell}'(t)|\leq -C\sqrt{n-d} |g_{\ell}(t)|\log(|g_{\ell}(t)|),
$$
Since $g_{\ell}$ is a definable function and $g_{\ell}(0)=0$, by exchanging the indices $i$ and $j$, if necessary, we may assume that $g_{\ell}(t)>0$ and $g'_{\ell}(t)>0$ for all small enough $t>0$.
By Lemma \ref{lem:estimate_derivative_log}, $\lim\limits_{t\to 0^+}\frac{g'_{\ell}(t)}{g_{\ell}(t)\log(g_{\ell}(t))}=-\infty$, which gives a contradiction. Therefore, $X$ is a $C^{1}$ smooth submanifold.
\end{remark}
\begin{remark}
Assume that $X$ is the graph of a $C^{1}$ smooth mapping $h\colon B\to \mathbb{R}^{n-d}$ for some open ball $B\subset \mathbb{R}^d$, but it is not necessarily definable. Note that it follows in the same way as in Claim \ref{claim:regularity} that, for a fixed integer number $k\geq 1$ and $\alpha\in [0,1]\cup \{\gamma-log-Lip;$ $\gamma\geq 0\}$, $h$ is $C^{k+1,\alpha}$ smooth if, and only if, $\eta\colon\mathcal{N}(X)\to X$ is a $C^{k,\alpha}$ smooth diffeomorphism. Thus, for a fixed integer number $k\geq 1$ and $\alpha\in [0,1]$, the mapping $\eta\colon\mathcal{N}(X)\to X$ is a homeomorphism such that $\eta^{-1}$ is $C^{k,\alpha}$ smooth if and only if $X$ is a $C^{k+1,\alpha}$ smooth submanifold.
\end{remark}

Thus, using the two remarks above, we can proceed similarly as in the proof of Theorem \ref{c11_smooth} to obtain the following log-Lipschitz version of this theorem.
\begin{theorem}\label{c11_smooth_log-Lip}
Let $X\subset \R^n$ be a locally closed set that is pure $d$-dimensional and locally definable in an o-minimal structure on $\R$. 
Suppose that $C_3(X)$ and $C_4(X)$ coincide linearly. Then, for a fixed nonnegative integer $k$, the mapping $\eta\colon\mathcal{N}(X)\to X$ is a homeomorphism such that $\eta^{-1}$ is $C^{k,log-Lip}$ smooth if and only if $X$ is a $C^{k+1,log-Lip}$ smooth submanifold.
\end{theorem}

The following examples show that the hypotheses of Theorem \ref{c11_smooth_log-Lip} are sharp.
\begin{example}
Let $C=\{(x,y)\in \C^2;x^2=y^3\}$. We have that $\eta\colon \mathcal{N}(C)\to C$ is a homeomorphism such that $\eta^{-1}$ is $C^{0,\alpha}$ smooth for some $\alpha\in (0,1)$. However, $C$ is not a $C^1$ submanifold.
\end{example}

In fact, we can obtain examples for any $\alpha\in (0,1)$.
\begin{example}
For a fixed $\alpha\in (0,1)$, let $k$ be a positive integer such that $\alpha<\frac{k}{k+1}$ and let $Y_k=\{(x,y)\in \R^2;x^4=y^{4k+2}\}$. We have that $C_3(Y_k,p)$ is a line for any $p\in Y_k$ and, moreover, $\eta\colon \mathcal{N}(Y_k)\to Y_k$ is a homeomorphism such that $\eta^{-1}$ is $C^{0,\alpha}$ smooth, but $X$ is not a $C^1$ submanifold.
\end{example}

We can even obtain examples where the Nash transformation is a bi-$\alpha$-H\"older homeomorphism for all $\alpha\in (0,1)$. 

\begin{example}\label{exam:holder_fails}
Let $f_{\pm}\colon\R\to \R$ be the functions defined as follows
$$
f_{\pm}(x)=\left\{\begin{array}{ll}
     \pm x^3e^{-\frac{1}{x}}&\text{if } x>0, \\
     0& \text{if } x\leq 0.
\end{array}\right.
$$
Let $X$ be the union of the graphs of $f_{-}$ and $f_{+}$ (see Figure \ref{fig:thm_fails_holder}). Then $C_3(X)$ and $C_4(X)$ coincide linearly, and $\eta\colon \mathcal{N}(X)\to X$ is a homeomorphism such that $\eta^{-1}$ is $C^{0,\alpha}$ smooth for all $\alpha\in (0,1)$, but $X$ is not even a topological manifold.
\end{example}

\begin{figure}[H]
\centering
\begin{tikzpicture}[scale=2]
  \draw[->] (-1.2,0) -- (2.2,0) node[right] {$x$};
  \draw[->] (0.35,-1.2) -- (0.35,1.2) node[above] {$y$};

  \draw[domain=0.35:1.35, smooth, variable=\x, red, thick] 
    plot ({\x}, {\x*\x*\x*exp(-1/\x)}) node[right] {$Graph(f_+)$};

  \draw[domain=0.35:1.35, smooth, variable=\x, red, thick] 
    plot ({\x}, {-\x*\x*\x*exp(-1/\x)}) node[right] {$Graph(f_-)$};

  \draw[domain=-1.2:0.35, smooth, red, thick] plot (\x, 0);

  \node[text=red] at (0.7,0.3) {$X$};
\end{tikzpicture}
\caption{Theorem \ref{c11_smooth_log-Lip} fails for H\"older regularity}\label{fig:thm_fails_holder}
\end{figure}

In the same way, for any $\gamma>1$,
we can obtain examples where the Nash transformation is a homeomorphism such that its inverse is $\gamma$-log-Lipschitz. 

\begin{example}\label{exam:gamma-log-Lip_fails}
Let $m$ be a positive integer number. Let $h_{\pm}\colon \R\to \R$ be the functions defined as follows: $h_{\pm}(x)=\pm x^3e^{-\frac{1}{x^m}}$.
Let $A$ be the union of the origin with the graphs of $h_{-}$ and $h_{+}$. Then $C_3(A)$ and $C_4(A)$ coincide linearly, and $\eta\colon \mathcal{N}(A)\to A$ is a homeomorphism such that $\eta^{-1}$ is $\gamma$-log-Lipschitz for any $\gamma>1+\frac{1}{m}$, but $X$ is not even a topological manifold.
\end{example}
We can even require that the set be a topological manifold. 
\begin{example}\label{exam:gamma-log-Lip_fails_top_manifold}
Let $m$ be a positive integer number. Let $g_{\pm}\colon (0,+\infty)\to \R$ be the functions defined as follows: $g_{\pm}(x)=\pm x^3e^{-\frac{1}{x^m}}$.
Let $Y$ be the union of the origin with the graphs of $g_{-}$ and $g_{+}$. Then $C_3(Y)$ and $C_4(Y)$ coincide linearly, and $\eta\colon \mathcal{N}(Y)\to Y$ is a homeomorphism such that $\eta^{-1}$ is $\gamma$-log-Lipschitz for any $\gamma>1+\frac{1}{m}$. Note also that $Y$ is a topological manifold, but it is not a $C^1$ submanifold.
\end{example}
\begin{figure}[H]
\centering
\begin{tikzpicture}[scale=2]
  \draw[->] (-1.2,0) -- (2.2,0) node[right] {$x$};
  \draw[->] (0.45,-1.2) -- (0.45,1.2) node[above] {$y$};

  \draw[domain=0.45:1.35, smooth, variable=\x, blue, thick] 
    plot ({\x}, {\x*\x*\x*exp(-1/\x*1/\x)}) node[right] {$Graph(g_+)$};

  \draw[domain=0.45:1.35, smooth, variable=\x, blue, thick] 
    plot ({\x}, {-\x*\x*\x*exp(-1/\x*1/\x)}) node[right] {$Graph(g_-)$};

\draw[smooth, red, thick] plot (0, 0);
  \node[text=blue] at (0.7,0.3) {$Y$};
\end{tikzpicture}
\caption{Theorem \ref{c11_smooth_log-Lip} fails for $\gamma$-log-Lipschitz regularity, $\gamma>1$}\label{fig:thm_fails_log-Lip}
\end{figure}

The condition that $C_3(X)$ and $C_4(X)$ coincide linearly cannot be dropped.
\begin{example}
Let $Z=\{(x,y)\in \R^2; y^2-x^3\leq 0\}$. We have that $\eta\colon \mathcal{N}(Z)\to Z$ is a $C^{\infty}$ diffeomorphism. Indeed, $\eta^{-1}(x)=(x,\R^2)$ for all $x\in Z$. However, $Z$ is not a $C^1$ submanifold. 
\end{example}

\subsection{Metric version of the Theorem of Nobile}

\begin{theorem}\label{thm:metric_nobile}
Let $X\subset \mathbb{C}^n$ be a pure dimensional complex analytic set. Then, $X$ is analytically smooth if and only if $\eta\colon \mathcal{N}(X)\to X$ is a homeomorphism that is locally bi-Lipschitz.
\end{theorem}
\begin{proof}
It is clear that if $X$ is smooth, then $\eta\colon \mathcal{N}(X)\to X$ is a homeomorphism that is locally bi-Lipschitz.

Reciprocally, assume that $\eta\colon \mathcal{N}(X)\to X$ is a homeomorphism that is locally bi-Lipschitz. By Theorem \ref{c11_smooth}, $X$ is a $C^{1,1}$ submanifold. By a result proved by Milnor in \cite[Remark, pp. 13-14]{Milnor:1968}, $X$ is analytically smooth. This also follows from the Lipschitz Regularity Theorem \cite[Theorem 4.2]{Sampaio:2016} (see also \cite{Sampaio:2015}).
\end{proof}

Note that the condition that $\eta\colon \mathcal{N}(X)\to X$ is locally bi-Lipschitz cannot be dropped. Indeed, if $X=\{(x,y)\in \C^2; x^2=y^3\}$, then $\eta\colon \mathcal{N}(X)\to X$ is a homeomorphism, but $X$ is not smooth at $0$.

Since any complex analytic set is locally definable in a polynomially bounded o-minimal structure on $\R$, by Corollary \ref{cor:c11_smooth_holder} and Theorem \ref{thm:metric_nobile}, we obtain the following.
\begin{corollary}\label{cor:metric_nobile_hoelder}
Let $X\subset \mathbb{C}^n$ be a pure dimensional complex analytic set. Then, $X$ is analytically smooth if and only if $\eta\colon\mathcal{N}(X)\to X$ is a homeomorphism such that $\eta^{-1}$ is $\alpha$-H\"older for all $\alpha\in (0,1)$.
\end{corollary}

\subsection{Real version of the Theorem of Nobile}

\begin{theorem}\label{thm:real_analytic_ck_smooth}
Let $X\subset \R^n$ be a pure $d$-dimensional real analytic set, and let $k$ be a nonnegative integer number. Then $X$ is a $C^{k+1,1}$ submanifold if, and only if, the mapping $\eta\colon\mathcal{N}(X)\to X$ is a homeomorphism such that $\eta^{-1}$ is $C^{k,1}$ smooth.
\end{theorem}
\begin{proof}
It is clear that if $X$ is a $C^{k+1,1}$ submanifold, then $\eta\colon\mathcal{N}(X)\to X$ is a homeomorphism such that $\eta^{-1}$ is $C^{k,1}$ smooth.

Reciprocally, assume that $\eta\colon\mathcal{N}(X)\to X$ is a homeomorphism such that $\eta^{-1}$ is $C^{k,1}$ smooth.

Since $\eta\colon\mathcal{N}(X)\to X$ is in particular a bijection, we obtain that $C_4(X,p)$ is a $d$-dimensional linear subspace for all $p\in X$. In particular, $C_4(X,p)$ continuously varies on $p$.

We are going to prove that $C_3(X,p)= C_4(X,p)$ for all $p\in X$. Suppose by contradiction that this is not true, that is, suppose that there is $p\in X$ such that $C_3(X,p)\subsetneq C_4(X,p)$. Then $W:=int(C_4(X,p)\setminus C_3(X,p))$ is not an empty cone.

We may assume that $p=0$ and we choose linear coordinates $(x,y)$ in $\R^n$ such that $C_4(X,0)=\{(x,y)\in\R^n; y=0\}$. Let $\pi \colon \R^n\to C_4(X,0)\cong \R^d$ be the orthogonal projection.
Since $X$ is a real analytic set, the multiplicity $\mod 2$ is well-defined. Let $U$ be a sufficiently small neighbourhood of $0$ in $X$ such that there is a subanalytic set $\sigma\subset C_4(X,0)$ such that $\#( \pi^{-1}(v)\cap U)\, (\mod 2)$ is constant for all sufficiently small $v\in C_4(X,0)\setminus \sigma$.

For $v\in W$ and a sufficiently small neighbourhood $U$ of $0$ in $X$, we have that $\pi^{-1}(tv)\cap U=\emptyset$ for all sufficiently small $t>0$. This shows that the multiplicity of $X$ at the origin is $0\,  (\mod 2)$.
Since $C_4(X,q)$ continuously varies on $q$, by shrinking $U$, if necessary, we may assume that $\delta(C_4(X,0),C_4(X,q))<1/3$ for all $q\in U$.  In particular, for each $q\in U$, the restriction of $\pi$ to $C_4(X,q)$ is a linear isomorphism. Therefore, $\pi (U\setminus {\rm Sing}(X))$ is an open subset of $C_4(X,0)$. We may assume that $U=X\cap \mathbf B^{n}(0,r)$ for some $r>0$.

Let $B=\pi(U)$, $B'=B\setminus \pi({\rm Sing}(X))$, $U'=\pi^{-1}(B')\cap U$, $N=\sup\limits_{x\in B'} \#\pi|_{U'} ^{-1}(x)$ and $S= \{ x\in B': \#\pi|_{U} ^{-1}(x)= N\}$.  By shrinking $r$, if necessary, we assume that $0\in \overline{S}$ and $\pi^{-1}(0)\cap U=\{0\}$. We have that $S$ is an open set, $\pi_U^{-1}(S)$ has exactly $N$ connected components, say $X_1,...., X_N$, and for each $i$, $X_i$ is the graph of a $C^1$ mapping $f_i\colon S\to \R^{n-d}$. By the assumptions on the tangent cones, we also have that each $f_i$ has bounded derivative. Let $\lambda\geq 1$ be a number such that $\|Df_i\|\leq \lambda$ in $S$ for all $i\in \{1,...,N\}$.

Since $0\in \overline{S}$, there is a definable $C^1$ smooth arc $\gamma\colon [0,\epsilon)\to \overline {S}$ such that $\gamma(0)=x$, $\gamma((0,\epsilon))\subset S$ and $\|\gamma(t)\|= t$ for all $t\in [0,\epsilon)$. 
Let $\tilde \gamma_i(t)=(\gamma(t),f_i(\gamma(t)))$ and $\tilde \gamma_j(t)=(\gamma(t),f_j(\gamma(t)))$. 

Since $m(X,0)=0\,(\mod 2)$, we have $N\geq 2$ and, in particular, $\pi|_{U}$ is not injective. 

By proceeding as in the proof of Theorem \ref{c11_smooth}, there is a constant $\tilde K$ such that $\delta(T_{(z,f_i(z))}X_i,T_{(z,f_j(z))}X_j)\geq \tilde K\|(Df_i)_z-(Df_j)_z\|$ for all $z\in S$.
Then, 
$\delta(T_{\gamma_i(t)}X_i,T_{\gamma_j(t)}X_j)\geq \tilde K\|(f_i\circ \gamma)'(t)-(f_j\circ \gamma)'(t)\|$ for all $t\in [0,\epsilon)$.

Since $\eta$ is, in particular, a local bi-Lipschitz homeomorphism, by proceeding in the same way as in the proof of Theorem \ref{c11_smooth}, we obtain
\begin{eqnarray*}
{\rm ord}_0 \|\gamma_i(t)-\gamma_j(t)\|&\leq & {\rm ord}_0 \|\gamma_i(t)-\gamma_j(t)\| - 1,
\end{eqnarray*}
which is a contradiction.
Therefore, $C_3(X,p)=C_4(X,p)$ for all $p\in X$. In particular, $C_3(X)$ and $C_4(X)$ coincide linearly. By Theorem \ref{c11_smooth}, $X$ is a $C^{k+1,1}$ submanifold. 
\end{proof}

Consequently, we obtain the following real version of the Theorem of Nobile.

\begin{theorem}\label{thm:real_analytic_smooth}
Let $X\subset \R^n$ be a pure $d$-dimensional real analytic set. Then the following statements are equivalent:
\begin{enumerate}
 \item [(1)] $X$ is a real analytic submanifold;
 \item [(2)] the mapping $\eta\colon\mathcal{N}(X)\to X$ is a real analytic diffeomorphism;
 \item [(3)] the mapping $\eta\colon\mathcal{N}(X)\to X$ is a $C^{\infty}$ diffeomorphism;
 \item [(4)] $X$ is a $C^{\infty}$ submanifold.
\end{enumerate}
\end{theorem}
\begin{proof}
It is clear that $(1)\Rightarrow (2)$ and $(2)\Rightarrow (3)$.

Let us prove $(3)\Rightarrow (4)$. 
So, assume that $\eta\colon\mathcal{N}(X)\to X$ is a $C^{\infty}$ diffeomorphism. In particular, $\eta^{-1}$ is $C^{k,1}$ smooth for any nonnegative integer $k$. By Theorem \ref{thm:real_analytic_ck_smooth} and the proof of Theorem \ref{c11_smooth}, $X$ is locally the graph of a mapping $h\colon B\to \R^{n-d}$ that is $C^{k+1,1}$ smooth for all nonnegative integer $k$. So, $h$ is $C^{\infty}$ smooth and, therefore, $X$ is a $C^{\infty}$ submanifold.

The implication $(4)\Rightarrow (1)$ follows from \cite[Proposition 1.1]{Ephraim:1973}.
\end{proof}

Since any real analytic set is locally definable in a polynomially bounded o-minimal structure on $\R$, by Corollary \ref{cor:c11_smooth_holder} and Theorem \ref{thm:real_analytic_ck_smooth}, we obtain the following.
\begin{corollary}\label{cor:real_analytic_ck_smooth_hoelder}
Let $X\subset \R^n$ be a pure $d$-dimensional real analytic set, and let $k$ be a nonnegative integer number. Then $X$ is a $C^{k+1,1}$ submanifold if, and only if, the mapping $\eta\colon\mathcal{N}(X)\to X$ is a homeomorphism such that $\eta^{-1}$ is $C^{k,\alpha}$ smooth for all $\alpha\in (0,1)$.
\end{corollary}

Note that for $k=0$, we cannot obtain an analogous result for $C^{0,\alpha}$ smoothness for any $\alpha\in(0,1)$.

\begin{example}
For each $k$, let $X_k=\{(x,y)\in \R^2; y^2=x^{2k+1}\}$. Note that $X_k$ is not $C^1$ smooth around the origin, and for $(x,y)\in X_k\setminus \{(0,0)\}$, we have that $\{(a,b)\in \R^2; \frac{(2k+1)}{2}y^{\frac{2k-1}{2k+1}}a+b=0\}$ is the tangent line to $X_k$ at $(x,y)$. Then the mapping $\eta\colon\mathcal{N}(X_k)\to X_k$ is a homeomorphism such that $\eta^{-1}$ is $C^{0,\alpha_k}$ smooth, where $\alpha_k=\frac{2k-1}{2k+1}$.
\end{example}

\end{document}